\documentclass[12pt,psamsfonts]{amsart}

\usepackage{amssymb} \theoremstyle{plain}
\usepackage[utf8]{inputenc}
\usepackage{graphicx}
\usepackage{tabulary}
\usepackage{tabularx}
\usepackage[all]{xy}
\usepackage{amssymb} 
\usepackage[shortlabels]{enumitem}

\usepackage[none]{hyphenat}
\newtheorem{theo}{Theorem}[section]
\newtheorem{cor}[theo]{Corollary} 

\newtheorem{exam}[theo]{Example}

\newtheorem{ques}[theo]{Question}

\newtheorem{prop}[theo]{Proposition}
 
\theoremstyle{definition} 
\newtheorem{defn}[theo]{Definition}

\begin{document}
\sloppy

\title[Compact-bounded topological groups]{Compact-bounded topological groups}

\author [J. A. Mart\'{\i}nez-Cadena]{J. A. Mart\'{\i}nez-Cadena}
\address[J. A. Mart\'{\i}nez-Cadena ]{Departamento de Matem\'aticas, Facultad de Ciencias, 
Universidad Nacional Aut\'onoma de M\'exico, Circuito Exterior s/n, Ciudad Universitaria, Apartado Postal 04510, Ciudad de M\'exico, M\'exico. } \email {martinezcadenajuan@gmail.com}


\begin{abstract}
We will study two subclasses of the class of feebly compact spaces in the class of (para)topological groups, the compact-bounded and weakly compact-bounded spaces, both introduced by J.  Angoa, Y.  F.  Ortiz-Castillo and A. Tamariz-Mascar\'ua in \cite{AOT}. Also,  we will study the $r$-weakly compact-bounded subsets of a topological space $X$.
\end{abstract}




\maketitle

\section{Introduction}

All topological spaces are assumed to be Hausdorff unless a stronger separation axiom is specified. A space is feebly compact if every locally finite family of non-empty open sets is finite. In the class of Tychonoff spaces, feeble compactness is equivalent to pseudocompactness and throughout this paper a pseudocompact space is always assumed to be Tychonoff. As usual, $cl_X(A)$ (or $\overline{A}$) denotes the closure of a subset $A$ of a topological space $X$.  We denote by $\omega$ the collection of finite ordinal numbers.  All undefined terms can be found in \cite{AT} or \cite{E}.
 
\begin{defn}\label{comboun}Let $X$ be a topological space.
\begin{enumerate}

\item $X$ is \textit{compact-bounded} if whenever $\mathcal U$ is a  countable family of non-empty open subsets of $X$, there is some compact subset $K\subset X$ which meets each element of $\mathcal U$.

\smallskip

\item $X$ is \textit{weakly compact-bounded} if whenever $\mathcal U$ is an infinite  countable family of  non-empty open subsets of $X$, there is a compact subset  $K\subset X$ which intersects infinitely many elements of $\mathcal{U}$. 
\end{enumerate}
\end{defn}

The compact-boundedness and weak compact-boundedness   properties were introduced by J.  Angoa, Y.  F.  Ortiz-Castillo and A. Tamariz-Mascar\'ua in the class of Tychonoff spaces  under the name \textit{pseudo-$\omega$-boundedness} and \textit{weak pseudo-$\omega$-boundedness}, respectively, in \cite{AOT}.

In this paper we study  compact-bounded and weakly compact-bounded spaces in the class of topological and paratopological groups. In Section 2, we study basic properties of these spaces in the class of topological spaces. In Section 3, we focus on finding conditions under which a paratopological group is a topological group, for this, we find a relation of compact-bounded spaces with the 2-pseudocompact spaces.

Given a subgroup $H$ of a topological group $G$,
in Section 4, we study how to extend the properties of the quotient space $G/H$ to $G$, this, in the cases  when $H$ is either a compact, locally compact or a (weakly) compact-bounded  subgroup. In Section 5, we find a condition on a Tychonoff space $X$ guaranteeing that the free topological group $F(X)$ and the free abelian topological group $A(X)$ admit an open continuous homomorphism onto the circle group $\mathbb{T}$.

Finally, in the Section 6, we study the $r$-weakly compact-bounded subsets of a topological space $X$ and relate them with the strongly $r$-pseudocompact, $r$-pseudocompact and $C$-compact subsets of $X$.

\section{Basic properties}
In this section, we present basic properties of (weakly) compact-bounded spaces. The next fact shows that the family $\mathcal{U}$ in the Definition \ref{comboun}-(2) can be considered as a family of  mutually disjoint non-empty open sets of the topological space $X$.

\begin{prop}[Proposition 2.4 in \cite{AMW}]\label{disj} For any topological space X,  the following are equivalent.
\begin{enumerate}
\item X is weakly compact-bounded;

\item For any infinite countable family of mutually disjoint  non-empty open subsets of $X$, there is a compact subspace  $K\subset X$ which intersects infinitely many elements of $\mathcal{U}$. 
\end{enumerate}
\end{prop}

A topological space with  property (2) of the previous proposition is called \textit{weakly cellular-compact} (\cite{AMW}).

Recall that a space is \textit{weakly regular} if every non-empty open set contains a non-empty regular closed subset.  Using Proposition \ref{disj}, the proof of the next result is identical to that of  Proposition 4.9 of \cite{PTV},    in which complete regularity is assumed, but weak regularity is all that is required.

\begin{prop}\label{fk} A weakly regular $k$-space $X$ is weakly compact-bounded if and only if $X$ is feebly compact. 
\end{prop}

\begin{cor}\label{ksp} If a topological group G is k-space, then $G$ is weakly compact-bounded if and only if $G$ is pseudocompact.
\end{cor}

\begin{ques} Let G be a feebly compact paratopological group. If G is k-space (or Fr\'echet), Is G weakly compact-bounded?
\end{ques}

The following proposition is immediate from the definition of compact-boundedness.

\begin{prop}\label{isolated} If $X$ is compact-bounded, then the closure of a countable set of isolated points of $X$ is compact; in particular, if a compact-bounded space has a countable dense set of isolated points, then it is compact.
\end{prop}

 The following implications trivially hold.
 
\smallskip		

\begin{center}
compact-bounded $\Longrightarrow$ weakly compact-bounded $\Longrightarrow$ feebly compact.
\end{center}

\smallskip

The following examples illustrate the differences between these spaces.

\begin{exam} There is a pseudocompact topological group  not weakly compact-bounded. 
\end{exam}

\begin{proof} Let $G$ be the topological group described in  Example 4.5 in \cite{Tka}.  Then, $G$ has the following properties:
\begin{enumerate}
\item $|G|=w(G)=2^{\omega}$;

\item $G$ is Boolean and pseudocompact;

\item  all countable subgroups of $G$ are $h$-embedded in $G$ (hence, closed in $G$);

\item $G$ has no non-trivial convergent sequences.
\end{enumerate}
As a consequence of (1)–(4), all compact subsets of $G$ are finite. Clearly, the group G fails to be weakly compact-bounded.
\end{proof}



\begin{exam}\label{Mr} There is a Tychonoff  weakly compact-bounded space that is not compact-bounded.
\end{exam}
\begin{proof} If $X=\Psi(\mathcal{A})$ is the Isbell-Mr\'owka space (where $\mathcal{A}$ is a maximal infinite almost disjoint family of infinite subsets of $\omega$), then  $X$ is locally compact and pseudocompact, thus it is weakly compact-bounded space (Proposition \ref{fk}).  But $\omega$ is a   countable dense subset of isolated points of $X$,  thus $X$ is not compact-bounded (Proposition \ref{isolated}). 
\end{proof}

\begin{ques} Is there a weakly compact-bounded  paratopological group that is not compact-bounded?
\end{ques}

\begin{prop}\label{imgreg} Suppose that X is a (weakly) compact-bounded space. Then 
\begin{enumerate}
\item any continuous image of X is (weakly) compact-bounded.

\item any regular closed subset of X is (weakly) compact-bounded.
\end{enumerate}
\end{prop}
\begin{proof} For compact-bounded spaces, items (1) and (2) have been earlier established in Corollary 3.6 of \cite{AOT}.  Now, for weakly compact-bounded spaces, the proof of (1) is trivial.  We will show (2). Let $A$ be a regular closed subspace of $X$ and let $\{U_n \cap A: n\in \omega \}$ be a  disjoint family of non-empty open subsets of $A$, where $U_n$ is open in $X$ for each $n\in \omega$.
Since $\{ U_n \cap{\rm  int}(A):n\in \omega\}$ is a disjoint family of non-empty open subsets of $X$, then there is a compact subspace $K\subset X$ which intersects infinitely many elements of $\{ U_n \cap {\rm int}(A):n\in \omega\}$. Thus,   $A$ is weakly compact-bounded.
\end{proof}

The clause (1) of the following proposition has a simple proof which we omit. The clause (2) is proved in Corollary 3.6 of \cite{AOT}. 

\begin{prop}\label{finden} Let X be a space.
\begin{enumerate}
\item If X is a finite union of (weakly) compact-bounded spaces, then X is (weakly) compact-bounded.

\item If X contains a dense (weakly)  compact-bounded space, then X is (weakly) compact-bounded.
\end{enumerate}
\end{prop}

\begin{prop}\label{subCB} For any topological space X, the following are equivalent:
\begin{enumerate}
\item X is weakly compact-bounded;

\item For any  family $\mathcal{U}$ of non-empty open subsets of X, there is a weakly compact-bounded subspace $A\subset X$ such that $A$ meets infinitely many elements  of $\mathcal{U}$. 
\end{enumerate}
\end{prop}
\begin{proof}(1)$\Rightarrow$(2) is clear since a compact space is  weakly compact-bounded. 

(2)$\Rightarrow$(1)  Let $\mathcal{U}=\{U_n:n\in \omega \}$ be a family  of  non-empty open subsets of $X$ and let $A$ be a weakly compact-bounded subspace of $X$ such that $A$ meets infinitely-many elements  of $\mathcal{U}$, that is, there is an infinite subset $I\subset \omega$  such that $A\cap U_m\neq \emptyset$ for each $m\in I$. Since $\{A\cap U_m:m\in I\}$ is a  family  of  non-empty open subsets of $A$, there is a compact subset $K_A$ in $A$ (thus in $X$) such that $\emptyset \neq K_A\cap A\cap U_m \subset K_A\cap U_m $ for each $m\in I$, that is, $K_A$ meets infinitely many elements  of $\mathcal{U}$.
\end{proof}

The next fact has a proof similar to the previous proposition.

\begin{prop}\label{subCB} For any topological space X, the following are equivalent:
\begin{enumerate}
\item X is compact-bounded;

\item For any infinite family $\mathcal{U}$ of non-empty open subsets of X, there is a  compact-bounded subspace $A\subset X$  such that $A$ meets every element  of $\mathcal{U}$. 
\end{enumerate}
\end{prop}

A space $X$ is said to have a \textit{regular $G_\delta$-diagonal} if the diagonal $\Delta_X=\{(x,x): x\in X \}$  can be represented as the intersection of the closures of a countable family of open neighbourhoods of $\Delta_X$ in $X\times X$.

We end this section with the following fact about (weakly) compact-bounded  and first-countable subsets of  topological groups.

\begin{prop} Let G be a topological group.  Then,  the following assertions are equivalent.
\begin{enumerate}
\item every (weakly) compact-bounded  subset of G is first-countable;

\item every (weakly) compact-bounded subset of G is metrizable.
\end{enumerate}
\end{prop}
\begin{proof} Clearly, it will suffice to show the sufficiency.  Let $A$ be a weakly compact-bounded (compact-bounded) subset of $G$, and let
$g: G\times G \to G$ be a map defined by $g(x,y)=x^{-1}y$ for every $x,y \in G$ (we note that $g$ is a continuous mapping). By Theorem 4.4 in \cite{AMW},  we have that $A\times A$ is weakly compact-bounded (compact-bounded), hence $g(A\times A)=A^{-1}A$ is weakly compact-bounded (compact-bounded), and thus $A^{-1}A$ is first countable.

Since the neutral element $e$ is in $ A^{-1}A$, there is a family $\{W_n : n\in \omega\}$ of non-empty open neighborhoods of $e$ in $A^{-1}A$ such that $\{ e\}=\bigcap \{W_n : n\in \omega\}=\bigcap \{cl_{ A^{-1}A}(W_n) : n\in \omega\}$. If $f=g|(A\times A)$, then $\Delta_A =f^{-1}(\{ e\})$. Now, we have 

$$
f^{-1}(\{ e\})=\bigcap \{ f^{-1}(W_n): n\in \omega\} \subseteq \bigcap \{cl_{ A\times A} (f^{-1}(W_n)): n\in \omega\} \subseteq 
$$
$$
\subseteq \bigcap  \{ f^{-1}(cl_{A^{-1}A}(W_n)): n\in \omega\}= f^{-1} \left( \bigcap  \{ cl_{A^{-1}A}(W_n): n\in \omega\} \right).
$$
\medskip

 On the other hand,  $f^{-1}(\{ e\})=f^{-1}(\bigcap  \{ cl_{A^{-1}A}(W_n): n\in \omega\})$. Thus, $\Delta_A = f^{-1}(\{ e\})= \bigcap \{cl_{ A\times A} (f^{-1}(W_n)): n\in \omega\}$,  that is, $A$ has a regular $G_\delta$-diagonal. By Corollary 4 in \cite{AB}, we conclude that $A$ is compact and metrizable.
\end{proof}

\section{When a  (weakly) compact-bounded paratopological group is a topological group}

The following examples illustrate that not every weakly compact-bounded paratopological group  is a topological group:

\begin{exam}There exists a  weakly compact-bounded paratopological group  which is not a topological group.
\end{exam}
\begin{proof} The paratopological group $H$ described in the Example 2.7.10 in \cite{lib} is  feebly compact and second countable, thus it is weakly compact-bounded (Proposition \ref{fk}). However, it is not a topological group.
\end{proof}

For a paratopological group $(G,\tau)$, the family $\{U^{-1}: U\in \tau \}$ defines a topology $\tau^{-1}$, named the \textit{conjugate} of $\tau$ and the pair $(G,\tau^{-1})$ is a paratopological group homeomorpic to $(G,\tau)$.

We will denote by  $\Delta_G $ the diagonal of the direct product $(G,\tau^{-1})\times (G,\tau)$.  By Lemma 2.2 in \cite{AS}, we know that if $G$ is a paratopological group, then the diagonal $\Delta_G$  of the product $(G, \tau^{-1} )\times (G,\tau)$ is a topological group.  Moreover,  $\Delta_G$ is closed in $(G,\tau)\times (G,\tau^{-1})$ whenever  $(G,\tau)$ is a paratopological group satisfying the $T_1$ separation axiom.

\begin{prop}  Let G be a paratopological group with topology $\tau$. If  the diagonal $\Delta_G \subset (G, \tau^{-1}) \times (G,\tau)$ is (weakly) compact-bounded, then $G$ is a (weakly) compact-bounded topological group.
\end{prop}
\begin{proof} Suppose that the subgroup $\Delta_G$ of $(G, \tau^{-1}) \times (G, \tau)$ is (weakly) compact-bounded. Since $G$ is a paratopological group, we have that $\Delta_G$ is a topological group (Lemma 2.2 in \cite{AS}), furthermore, since $\Delta_G$ is (in particular) feebly compact, then $\Delta_G$ is pseudocompact, by Lemma 2.3 in \cite{AS}, we conclude that $G$ is a topological group. 

Now, since $G$ is a topological group, clearly $\tau = \tau^{-1}$, 
whence it follows that $\Delta_G$ is a topological group topologically isomorphic to $G$. Hence, $G$  is (weakly) compact-bounded.
\end{proof}

A paratopological group $G$ is \textit{2-pseudocompact} if  whenever $\{U_n :n\in \omega\}$ is a decreasing sequence of non-empty open subsets of $G$, we have $\bigcap \{ cl_G(U_n^{-1}):n\in \omega \}   \neq \emptyset$.

\begin{prop} Let $G$ be a compact-bounded paratopological group. If for every compact subset $K\subset G$, the set  $K^{-1}$ is countably compact in $G$, then $G$ is 2-pseudocompact.
\end{prop}
\begin{proof} Let $\{U_n :n\in \omega\}$ be a decreasing sequence of non-empty open subsets of $G$. Then, there is a compact set $F$ of $G$ such that $F\cap U_n \neq \emptyset $ for each $n\in \omega$.  We note that since $F$ is compact in $G$,  we have that $cl_G(F^{-1})=F^{-1}$ (Lemma 2.3.5 in \cite{AT}).  It follows that  $\emptyset \neq  F^{-1}\cap U_n^{-1} \subset cl_G({F^{-1}\cap U_n^{-1}}) \subset cl_G({F^{-1}})\cap cl_G({U_n^{-1}}) \subset F^{-1}\cap cl_G({U_n^{-1}})$
 for each $n\in \omega$. Clearly,  the family $\{ F^{-1}\cap cl_G(U_n^{-1}) : n\in \omega\}$ is a decreasing sequence of non-empty closed subsets of the countably compact set $K^{-1}$. Thus,  we can conclude that
$$
\emptyset \neq \bigcap\{ F^{-1}\cap cl_G(U_n^{-1}): n\in \omega \}\subset \bigcap   \{ cl_G(U_n^{-1}): n\in \omega \}.
$$
Therefore, $G$ is 2-pseudocompact.
\end{proof}

In \cite{BR}, a semitopological group $G$ is \textit{left $\omega$-precompact} if for each neighborhood $U$ of the identity of $G$ there is a countable subset $F$ of $G$ such that $FU=G$.  Recall that a paratopological group $G$ is \textit{topologically periodic} if for each $x\in G$ and a neighborhood $U$ of the identity of $G$ there is $n\geq 1$ such that $x^n \in U$.  Since every 2-pseudocompact paratopological group is a Baire space (Theorem 2.2 in \cite{AS}),  the Propositions 3.15 and 3.24 in \cite{BR} provide some conditions under which a  compact-bounded paratopological group is a topological group.

\begin{cor} Let G be a compact-bounded paratopological group such that the set $K^{-1}$ is countably compact in G for each compact subset $K\subset G$. If at least one of the following condition holds:
\begin{enumerate}
\item G is left $\omega$-precompact;

\item G is topologically periodic;

\item G has countable pseudocharacter.

\end{enumerate}
Then, $G$ is a topological group.  
\end{cor}

\section{Quotients}

It is known that if $G$ is a topological group and $H$ is a compact subgroup of $G$, then the quotient mapping $\pi:G\to G/H$ is a perfect map, L. S.  Pontryagin used this fact in \cite{P}.  This fact helps to extend some compactness-like properties from $G/H$ to $G$, which are inverse invariants of perfect mappings, such as local compactness, countable compactness,  paracompactness, the Lindelöf property,  $\sigma$-compactness and \v{C}ech-completeness.

The Mrówka space $\Psi(\mathcal{D})$ (where $\mathcal{D}$ is an infinite maximal almost disjoint family on $\omega$) is a weakly compact-bounded space. The subspace $\mathcal{D}$ of $\Psi(\mathcal{D})$ is closed, discrete and infinite (thus, not pseudocompact). That is, weak compact-boundedness is not hereditary with respect to closed subspaces.

It follows, from Lemma \ref{imgreg}-(2) and Theorem 3.7.29 in \cite{E}, that this property is not inverse invariant of perfect mappings. Thus, it is not immediate  that if the group $G/H$ is weakly compact-bounded and $H$ is compact, then $G$ is weakly compact-bounded. However, the following fact shows a way to prove this.

\begin{prop}\label{quotient}  Let G be a topological group and H be a  compact subgroup of G such that the quotient space $G/H$ has some of the following properties:
\begin{enumerate}
\item G/H is weakly compact-bounded,

\item G/H is compact-bounded.
\end{enumerate}
Then, G also has the same property, that is, in case (1) G is weakly compact-bounded, in case (2) G is compact-bounded.
\end{prop}
\begin{proof} We will show (2) (the proof of (1)  is similar). Let $\{U_n :n\in \omega\}$ be a family of non-empty open sets in $ G$. Then, $\{\pi (U_n) :n\in \omega\}$ is a family of non-empty open sets in $G/H$, thus there is a compact set $E$ in $G/H$ such that $\pi(U_n) \cap E \neq \emptyset $ for each $n\in \omega$.  We note that the set $K=\pi^{-1}(E)$ is  a compact in $ G$, because $\pi:G\to G/H$ is a perfect map (Theorem 3.7.2 in \cite{E}). 

Since the map $\pi$ is surjective, we have that $ \pi(U_n) \cap E =\pi(U_n) \cap \pi(\pi^{-1}(E))=\pi(U_n) \cap \pi(K) \neq \emptyset $ for each $n\in \omega$. It follows that, for each $n\in \omega$, there are  $x_n\in U_n$ and $k_n \in K$, such that $\pi(x_n)=\pi(k_n)$, thus  $x_n \in k_n H\subset KH$ for each $n\in \omega$. Since $KH$ is compact in $ G$ (Proposition 1.4.31 in \cite{AT}), we have that the set $F=cl(\{x_n:n\in \omega\})$ is also compact in $G$ and clearly, $F\cap U_n \neq \emptyset $ for each $n\in \omega$.  Therefore,  $G$ is compact-bounded topological group.
\end{proof}

By Theorem 2.4 in \cite{A} and Proposition \ref{ksp}, we have:

\begin{cor} Let G be a topological group and H be a  compact subgroup of G such that $G/H$ is  a pseudocompact  k-space. Then G is weakly compact-bounded. 
\end{cor}

The next fact is an immediate consequence of Proposition 6.3-(b) in \cite{CR} and Proposition \ref{ksp}:

\begin{cor} Let a topological group G be a k-space and H be a closed, normal subgroup of G. If H and $G/H$ are pseudocompact spaces, then G is weakly compact-bounded.
\end{cor}

If $G$ is a topological group,  $H$ is a locally compact subgroup of $G$, and $\pi:G \to G/H$ is the natural quotient mapping, by Corollary 2.1 in \cite{A},  there is an open neighbourhood $U$ of the neutral element $e$ in $G$ such that the restriction $\pi | \overline{U}:  \overline{U}  \to \pi(\overline{U})$ is a perfect mapping.  Moreover,  $\pi(\overline{U})=\overline{\pi(U)}$, so $\pi(\overline{U})$ is a regular closed set in $G/H$.  

Assuming that the subgroup $H$ of a topological group $G$ is locally compact, some local properties of  compactness type on $G/H$ can be extended to $G$ (see Corollary 3.2.6 in \cite{AT}) using the fact that these properties are invariant under perfect mappings.  In the case of locally (weakly) compact-bounded spaces it is also possible as shown by the following fact:

\begin{prop}\label{loc} Let G be a topological group and H be a locally compact subgroup of G such that the quotient space $G/H$ has one of the following properties:
\begin{enumerate}
\item $G/H$ is weakly compact-bounded;

\item $G/H$ is compact-bounded.
\end{enumerate}
Then, G also has the same property, that is, in case (1) G is locally weakly compact-bounded, in case (2) G is locally compact-bounded.
\end{prop}
\begin{proof} We will show (2) (the proof of (1)  is similar).  As mentioned earlier, there is an open neighbourhood $U$ of the neutral element $e$ such that the restriction $\pi | \overline{U}:  \overline{U}  \to \pi(\overline{U})$ is a perfect mapping.  

Let $\{U_n :n\in \omega\}$ be a family of non-empty open sets in $ \overline{U}$. Then, $\{\pi (U_n) :n\in \omega\}$ is a family of non-empty open sets in $\pi(\overline{U})$, thus there is a compact $E$ in $\pi(\overline{U})$ such that $\pi(U_n) \cap E \neq \emptyset $ for each $n\in \omega$.  Using an argument analogous to that in the proof of Proposition \ref{quotient}, we can find a compact subset $F$ of $\overline{U}$ such that $F\cap U_n \neq \emptyset $ for each $n\in \omega$.  Thus,  $\overline{U}$ is a  compact-bounded neighbourhood of $e$.
\end{proof}




A topological property $\mathcal{P}$ is said to be a \textit{three space property} if for every topological group $G$ and a closed invariant subgroup $H$ of $G$, the fact that both groups $H$ and $G/H$ have $\mathcal{P}$ implies that $G$ also has $\mathcal{P}$.  

In the study of  weakly compact-bounded spaces and their properties related to quotients, the following question immediately arises:

\begin{ques}\label{q1} Is the (weak) compact-boundedness a three space property?
\end{ques}

To try to give a partial answer to the Question \ref{q1}, we need the following concept: Let $G$ be a topological group, let $H$ be a closed subgroup of $G$ and   $\pi:G \to G/H$ be the natural quotient mapping.  We say that $G/H$ is a \textit{$P_C$-space} if for every compact subset $C$ of  $G/H$, the inverse image $\pi^{-1}(C)$ is a compact subset of $X$.

\begin{prop} Let $G$ be a topological group and  $H$ a closed subgroup of $G$ such that  $G/H$ is a $P_C$-space. Suppose that $H$ and $G/H$ satisfy one of the following: 
\begin{enumerate}
\item $H$ and $G/H$ are weakly compact-bounded;

\item $H$ and $G/H$ are compact-bounded.
\end{enumerate}
Then, $G$ has this property.
\end{prop}
\begin{proof} We will  show (2) (the proof of (1)  is similar).  Let $\{U_n :n\in \omega\}$ be a family of non-empty open sets in $G$ and let $\pi:G \to G/H$ be the natural quotient mapping. Then,  $\{\pi (U_n) :n\in \omega\}$ is a family of non-empty open sets in $G/H$, thus there is a compact set $E$ in $G/H$ such that $\pi(U_n) \cap E \neq \emptyset $ for each $n\in \omega$.  Since $G/H$ is a $P_C$-space, the set $A=\pi^{-1}(E)$ is a compact subset of $G$. Clearly,  $ \pi(U_n)\cap \pi(A)\neq \emptyset$,  then there is $x_n\in U_n$ and $a_n\in A$ such that $\pi(x_n)=\pi(a_n)$ for each $n\in \omega$. Thus $x_n \in a_n H$, and consequently $a_n H \cap U_n \neq \emptyset$ for each $n\in \omega$. 

Now, we note that the set  $a_n H \cap U_n$ is homeomorphic to $ H \cap a^{-1}_n U_n$ for each $n\in \omega$.  Since $\{H\cap a^{-1}_n U_n :n\in \omega\}$ is a family of non-empty open sets in $H$, there is a compact subset $K$ of $H$ such that $K\cap (H\cap a^{-1}_n U_n)= K \cap a^{-1}_n U_n \neq \emptyset$ for each $n\in \omega$.  Thus, we have that $a_n K\cap  U_n\neq \emptyset$ for each $n\in \omega$, that is $A K\cap  U_n\neq \emptyset$ for each $n\in \omega$.  Since the set $AK$ is a compact subset of $G$ (Proposition 1.4.31 in \cite{AT}), we conclude that  $G$ is compact-bounded.
\end{proof}

Let $X$ and $Y$ be topological spaces. A topological property $\mathcal{P}$ will be called an \textit{inverse fiber property} if whenever $f:X\to Y$ is a continuous onto mapping such that the space $Y$ and the fibers of $f$ have $\mathcal{P}$, then $X$ also has $\mathcal{P}$.  By Proposition 2.2 in \cite{BT2}, we know that every inverse fiber property is a three space property.

\begin{prop} To have all (weakly) compact-bounded  subsets finite is a three space property.
\end{prop}
\begin{proof} Let $\mathcal{P}_{cb}$ be the property that all compact-bounded subsets of a space $X$ are finite. It will be sufficient to prove that $\mathcal{P}_{cb}$ is an inverse fiber property.

Let $f:X \to Y$ be a continuous onto map. Suppose that $Y$ and the fibers of $f$ have $\mathcal{P}_{cb}$. Just as in the proof of Proposition 2.4(a) in \cite{BT2}, let $A\subset X$ be a compact-bounded subset of $X$. Then $f(A)\subset Y$ is compact-bounded subset of $Y$ and thus $f(A)$ is finite. Thus, the family $\{A \cap f^{-1}(y):y\in f(A) \}$ is a finite partition of $A$ into disjoint clopen subsets of $A$. Then, each set $A \cap f^{-1}(y)$ is compact-bounded by (2) of Theorem \ref{imgreg}; thus $A\cap f^{-1}(y)$ is finite. Therefore, $A$ is finite. 
\end{proof}

As usual, for a topological space $X$, $A(X)$ is the Abelian free topological group generated by $X$, and for a locally compact topological space $X$, $\alpha X$ is the one point compactification of $X$. In contrast to the previous result, we have the following example:

\begin{exam}[Example 4.6 in \cite{BT2}] ``To have all (weakly) compact-bounded  subsets  compact" is not a three space property.
\end{exam}
\begin{proof} Let $X=\Psi(\mathcal{D})$ be the space described in  Example \ref{Mr}. Then $X$ is weakly compact-bounded but it is not compact.  Let $f:\alpha X \to \omega +1$ be the mapping defined by $f(n)=n$ for each $n\in \omega$ and $f(x)=\omega$ for each $x\in \alpha X\setminus \omega$. Since $f$ is closed, it extends to a continuous open homomorphism $\bar{f} :A(\alpha X) \to A(\omega +1)$. Let $G$ be the subgroup of $A(\alpha X)$ generated by $X$, $N=\mbox{ker} \bar{f}$ and $K=G \cap N$. By Example 4.6 in \cite{BT2},  we know that   $K$ is a closed subgroup of $G$ such that all pseudocompact subspaces of the groups $K$ and $G/K$ are compact (thus $K$ and $G/K$  have the required property) and that $X$ is a closed weakly compact-bounded, non-compact subset of $G$.
\end{proof}

\begin{ques} Consider the following properties:
\begin{enumerate}

\item All weakly compact-bounded (compact-bounded) subsets are closed;

\item All weakly compact-bounded (compact-bounded) subsets are first-countable;

\item All weakly compact-bounded (compact-bounded) subsets are metrizable.
\end{enumerate}
Is any of these properties a three space property?
\end{ques}

\section{Free topological groups}

Let $F(X)$ be the free topological group and $A(X)$ be the free abelian topological group of a Tychonoff space $X$. We denote the groups $F(X)$ or $A(X)$ by $G(X)$. Also, $\beta X$ denotes the Stone-\v{C}ech compactification of $X$.

Recall that a topological space $X$ is \textit{$\omega$-bounded}  if the closure of every countable subset of $X$ is compact.  Clearly, every $\omega$-bounded space is a compact-bounded space. 

\begin{exam} There is a  compact-bounded space that is not $\omega$-bounded.
\end{exam}

\begin{proof} The zero-dimensional non-$\omega$-bounded space described in Example 2.9 in \cite{JSS} has the property that every countable discrete subset of $X$ has compact closure in $X$. We will show that $X$ is compact-bounded. Let $\{ U_n :n \in \omega\}$ be a family of non-empty open sets in $X$. 
As we have already mention, we can assume that $\{ U_n :n \in \omega\}$ is a family of mutually disjoint sets. Now we take $x_n \in U_n$ for each $n < \omega$. So, the collection $\{x_n : n < \omega\}$ is countable and discrete. 
Thus, the set ${\rm cl}_X(\{ x_n:n\in \omega\})$ is a compact subset  of $X$ which intersects each element of $\{ U_n :n \in \omega\}$.  Therefore, $X$ is a compact-bounded space.
\end{proof} 

By Proposition 4.6 in \cite{LT},  the following conditions are equivalent for an $\omega$-bounded space $X$:

(1) The Stone-\v{C}ech compactification $\beta X$ is scattered;

(2) $X$ is scattered.

Clearly, (1) always implies (2).  The following fact shows that compact-boundedness also makes the above conditions equivalent.

\begin{prop} If X is a scattered compact-bounded Tychonoff space,  then $\beta X$ is scattered.
\end{prop}
\begin{proof} Suppose that $\beta X$ is not scattered. Then, there is a continuous surjection $f:\beta X \to I$, where $I$ denotes the closed unit interval (Theorem 8.5.4 in \cite{S}).    Let $\{V_n:n\in \omega\}$ be a countable base of $I$. Since  $\{f^{-1}(V_n)\cap X:n\in \omega\}$ is a family of non-empty open sets in $X$, there is a compact subset $K$ of $X$ such that $K\cap (f^{-1}(V_n)\cap X)=K\cap f^{-1}(V_n) \neq \emptyset$ for each $n\in \omega$. Thus, $
\emptyset\neq f(K\cap f^{-1}(V_n))\subset f(K)\cap f( f^{-1}(V_n))\subset f(K)\cap V_n, $
for each $n\in \omega$. This implies that the set $f(K)$ is dense in $I$. Since $K$ is compact and scattered, the image $f(K)=cl_I(f(K))=I$ is a countable set (Corollary 8.5.6 in \cite{S}), a contradiction.
\end{proof}

 With the previous fact we can extend the Theorem 4.7 in \cite{LT} to compact-bounded spaces.

\begin{prop} Let X be a compact-bounded space. Then the following conditions are equivalent:
\begin{enumerate}
\item The group $G(X)$ admits an open continuous homomorphism onto the circle group $\mathbb{T}$;

\item The group $G(X)$ admits a non-trivial metrizable quotient;

\item The group $G(X)$ admits a non-trivial separable metrizable quotient;

\item X is not scattered.
\end{enumerate}
\end{prop}
\begin{proof} Since $X$ is pseudocompact,  Theorem 4.5 in \cite{LT}, implies the equivalences  (1)$\Leftrightarrow$(2)$\Leftrightarrow$(3).  We note that (d) of Theorem 4.5 in \cite{LT} is equivalent to item (4), thus we can conclude this fact.
\end{proof}

\section{R-weakly compact-bounded  subsets}
In this section we consider only Tychonoff spaces.  

\begin{defn}\label{r-com} Let $A$  be a subset of a topological space $X$.
\begin{enumerate}

\item $A$ is  \textit{r-pseudocompact} in $X$ if every infinite family of open sets in $X$ meeting $A$ has an accumulation point in $A$.

\smallskip

\item $A$ is  \textit{strongly  r-pseudocompact} in $X$ if the product $A\times B$ is $r$-pseudocompact in $X\times Y$ whenever $B$ is a $r$-pseudocompact subset of a space $Y$. 

\smallskip

\item $A$ is \textit{r-weakly compact-bounded} in $X$ if for every infinite family $\mathcal{U}$ of open sets in $X$ meeting $A$,  there is some compact subset $K_A\subset A$ which intersects infinitely many elements of $\mathcal{U}$. 

\end{enumerate}
\end{defn}

A detailed study of properties (1) and (2) can be found in \cite{HST},  \cite{ST} and \cite{XY}.





\begin{prop}[Theorem 2.5 in \cite{HST}]\label{timespseudo} Let $A$ be a subset of a space $X$. Then, $A$ is  \textit{strongly  r-pseudocompact} in $X$ if and only if the product $A\times Y$ is $r$-pseudocompact in $X\times Y$ whenever $Y$ is a pseudocompact space $Y$. 
\end{prop}

\begin{prop}\label{wcbsrp}  Every r-weakly compact-bounded  subset of  a space X is strongly r-pseudocompact of X.
\end{prop}
\begin{proof} Let $A$ be a weakly compact-bounded subset of $X$ and  let $Y$ be a pseudocompact space.  By Proposition \ref{timespseudo}, it will suffice to show that  $A\times Y$ is $r$-pseudocompact in $X \times Y$.

Let $\{U_n \times V_n:n\in \omega\}$ be a  family of non-empty open subsets of $X\times Y$ where $\{U_n :n\in \omega\}$ is a family of non-empty open subsets of $X$ and $\{V_n :n\in \omega\}$ is a family of non-empty open subsets of $Y$ such that $(U_n \times V_n) \cap (A\times Y)  \neq \emptyset$ for each $n\in \omega$.   Since $ U_n \cap A \neq \emptyset$  for each $n\in \omega$, there is a compact subset $K_A$ of $A$ and an infinite subset $I\subset \omega$ such that $U_m \cap K_A \neq \emptyset$ for each $m\in I$.  By Corollary 3.10.27 \cite{E}, the product $K_A \times Y$ is pseudocompact (and thus a $r$-pseudocompact subset of $X\times Y$). Since the family $\{ U_m \times V_m:m\in I \}$ meeting $K_A \times Y$, then $\{ U_m \times V_m:m\in I \}$ has an accumulation point in $K_A \times Y\subset A\times Y$, which is also an accumulation point of the family $\{U_n \times V_n:n\in \omega\}$. Therefore,   $A\times Y$ is $r$-pseudocompact in $X \times Y$.
\end{proof}

\begin{exam} There is a r-weakly compact-bounded subset of a  space X  that it is not pseudocompact.
\end{exam}
\begin{proof} Let $X=\Psi(\mathcal{D})$ be the space described in the Example \ref{Mr}.  Since every point in $\omega$ is isolated in $X$,  $\mathcal{D}$ is $r$-pseudocompact in $X$. Also $\mathcal{D}$ is closed in $X$, thus it is locally compact, that is, $\mathcal{D}$ is  $r$-weakly compact-bounded  in $X$ (Proposition \ref{rpse}).  However, $\mathcal{D}$ is a closed discrete subspace of $X$,  and hence $\mathcal{D}$ is not pseudocompact.
\end{proof}

\begin{exam} There is an $r$-pseudocompact subset  of a space X  that it is not a r-weakly compact-bounded in $X$.
\end{exam}
\begin{proof} The dense pseudocompact subspace $P$ of $\beta \omega$ described in the Example 2.4-(4) in \cite{HST} is $r$-pseudocompact but it is not  strongly $r$-pseudocompact in $\beta \omega$  (Example 2.4-(5) in \cite{HST}) thus,  it is not $r$-weakly compact-bounded (Proposition\ref{wcbsrp}).
\end{proof}

\begin{ques} Is there a strongly $r$-pseudocompact subset of a space $X$ that is not $r$-weakly compact-bounded?
\end{ques}

A similar argument of the proof of Proposition 2.4 in \cite{AMW} proofs the following  fact.

\begin{prop}\label{rpseu} Let $A$ be a subset of a space $X$. Then,  the following conditions are equivalent:
\begin{enumerate}
\item  A is r-weakly compact-bounded in X;

\item For every infinite family $\mathcal{U}$ of mutually disjoint  open sets in $X$ meeting $A$,  there is some compact subset $K_A\subset A$ which intersects infinitely many elements of $\mathcal{U}$. 

\end{enumerate}
\end{prop}

The following fact shows sufficient conditions under which a subset $r$-pseudocompact of a topological space $X$ is a $r$-weakly compact-bounded subset.

\begin{prop}\label{rpse} An r-pseudocompact subset A of a space X is r-weakly compact-bounded in X provided A is a k-space.
\end{prop}
\begin{proof} Let $A$ be a $r$-pseudocompact $k$-space subset of $X$ and let $\{U_n:n\in \omega\}$ be a  family of mutually disjoint non-empty open subsets of $X$ that meet $A$.  By  regularity of $X$, there is open subsets $V_n$ of $X$ such that $V_n\cap A \subset \overline{V_n}\cap A \subset U_n\cap A$ for each $n\in \omega$.  Since $A$ is $r$-pseudocompact, the family $\{V_n:n\in \omega\}$ has an accumulation point $q \in A$, which is an accumulation point of the set $\bigcup\{\overline{V_n}\cap A: n\in \omega\}$. Since $\{U_n\cap A:n\in \omega\}$ is  a  family of mutually disjoint sets in $A$, it follows that $q\notin \bigcup\{\overline{V_n}\cap A: n\in \omega\}$, thus the set $ \bigcup\{\overline{V_n}\cap A: n\in \omega\}$ is not closed in $A$.  Since $A$ is a $k$-space,  there is a compact $K_A$ of $A$ such that $K_A \cap  \bigcup\{\overline{V_n}\cap A: n\in \omega\}$ is not closed in $K_A$, that is, the set $\{n\in \omega: K_A\cap \overline{V_n} \neq \emptyset \}$ is infinite and $K_A$ meets infinitely many elements  of $\{\overline{V_n}: n \in \omega\}$ (and thus of $\{U_n:n\in \omega\}$).  By Proposition \ref{rpseu}, we conclude that $A$ is an $r$-weakly compact-bounded subset of  $X$.
\end{proof}

By Proposition \ref{wcbsrp} and the previous proposition we have the next fact.

\begin{cor}\label{eqwr} Let X be a topological space. Then the following conditions are equivalent for a k-space subset A of X:
\begin{enumerate}
\item A is $r$-weakly compact-bounded in X;

\item A is strongly r-pseudocompact;

\item A is r-pseudocompact.
\end{enumerate}
\end{cor}

\begin{prop} Let $A_i$ be a $r$-weakly compact-bounded subset of a space $X_i$, where $i\in \omega$. Then, the product $\prod_{i\in \omega} A_i$ is $r$-weakly compact-bounded in $\prod_{i\in \omega} X_i$.
\end{prop}
\begin{proof} We denote by $A$ and $X$ the products $\prod_{i\in \omega} A_i$  and  $\prod_{i\in \omega} X_i$, respectively. Let $\mathcal{U}= \{U_m:m\in \omega \}$ be a family of  non-empty open subsets of $X$ that meets $A$. Then,  $\pi_0[\mathcal{U}]=\{\pi_0[U_m]:m\in \omega \}$ is a countable family of non-empty open sets in $X_0$ that meets $A_0$.  Thus,  there is a compact space $K_0\subset A_0$ which meets infinitely many elements of  $\pi_0[\mathcal{U}]$, say $\mathcal{U}_0=\{ U\in \mathcal{U}: K_0\cap \pi_0[U]\neq \emptyset\}$.  We can assume that $\pi_1[\mathcal{U}]$ is a countable family of non-empty open sets in $X_1$ that meets $A_1$, thus   there is a compact space $K_1\subset A_1$ which meets infinitely many elements of  $\pi_1[\mathcal{U}]$, say $\mathcal{U}_1=\{ U\in \mathcal{U}: K_1\cap \pi_1[U]\neq \emptyset\}$, iterating this argument, we obtain a decreasing sequence of infinite subsets $\{J_n:n\in \omega \}$ of $\omega$, a decreasing sequence $\{\mathcal{U}_n:n\in \omega \}$ of infinite subsets of $\mathcal{U}$ where $\mathcal{U}_n=\{ U_m:m\in J_n\}$ and compact subspaces $K_n\subset A_n$ such that $K_n$ meets each element of the family of non-empty open sets $\pi_n[\mathcal{U}]$. Let $J\subset \omega$ be an infinite set which is almost contained in each set $J_n$. Let $\mathcal{V}=\{ U_j \in \mathcal{U}:j\in J\}$.  For each $n\in \omega$ and each $m\in J\setminus J_n$ we can take a finite set $F_n\subset A_n$ which meets each $\mathcal{V}\setminus \mathcal{U}_n$ and then $E_n=K_n \cup F_n$ is a compact subspace which meets each element of the family $\pi_n[\mathcal{U}]$ for each $n\in J$ and so $\prod_{i\in \omega} E_i$ is the required compact subspace of $A$.
\end{proof}

\begin{cor} Let $A_i$ be a $r$-pseudocompact k-space subset of a  space $X_i$, where $i\in \omega$. Then, the product $\prod_{i\in \omega} A_i$ is $r$-weakly compact-bounded in $\prod_{i\in \omega} X_i$.
\end{cor}

A subset $A$ of a topological space $X$ is said to be \textit{C-compact} in $X$ if the image $f(A)$ is compact, for any continuous real-valued function $f$ on $X$.  

Now, the following implications are clear:
 
\smallskip		

\begin{center}
$r$-weakly compact-bounded $\Rightarrow$ strongly $r$-pseudocompact  $\Rightarrow$ $r$-pseudocompact  $\Rightarrow$ $C$-compact
\end{center}

\smallskip

The Example 2.4-(2) in \cite{HST} shows a example of a $C$-compact subset of a compact space not $r$-pseudocompact.

By Corollary 3.11 in \cite{HST}, we know that $C$-compactness and strong $r$-pseudocompactness coincide in the class of topological groups.   
Thus, a consequence of Corollary \ref{eqwr} is the following:

\begin{prop}\label{Cwcc} Let G be a  topological group and let $A$ be a k-space subset of G.  Then,  the following conditions are equivalent:
\begin{enumerate}
\item A is $r$-weakly compact-bounded in G;

\item A is strongly r-pseudocompact in G;

\item A is r-pseudocompact in G;

\item A is C-compact in G.

\end{enumerate}
\end{prop}

\begin{cor} Let $A_i$ be a $C$-compact k-space subset of a topological group $G_i$, where $i\in \omega$. Then, the product $\prod_{i\in \omega} A_i$ is $r$-weakly compact-bounded in $\prod_{i\in \omega} G_i$.
\end{cor}

If we want to know when the  $r$-weakly compact-bounded subsets of a paratopological group coincide with the $C$-compact  subsets, we have the class of  weakly $\omega$-admissible paratopological groups (\cite{XY}) that answer this question. By Theorem 3.20 in \cite{XY} and Corollary \ref{eqwr} the next fact is immediate:

\begin{prop} Let G be a  paratopological group. If G is weakly $\omega$-admissible, then the following conditions are equivalent for a  k-space subset A of G:
\begin{enumerate}
\item A is $r$-weakly compact-bounded in G;

\item A is strongly r-pseudocompact in G;

\item A is r-pseudocompact in G;

\item A is C-compact in G.

\end{enumerate}
\end{prop}

  We call a paratopological group $G$ \textit{totally $\omega$-narrow} if the topological group $G^*$ associated to $G$ is $\omega$-narrow.
  
  We say that a paratopological group $G$ is \textit{precompact} if for every neighborhood $U$ of the identity $e$ in $G$, there exists a finite subset $F$ of $G$ such that $G=FU=UF$. 

 It is possible to find some properties of paratopological groups under which a $C$-compact subset becomes a subset $r$-weakly compact-bounded.   By  Corollary 4.3 in \cite{ST} and Proposition \ref{rpse},  we have the next fact:

\begin{prop} Let A be a C-compact k-space subset of a Tychonoff paratopological group G which satisfies one of the following conditions:
\begin{enumerate}
\item G is totally $\omega$-narrow;
\item G is commutative and has countable Hausdorff number; 
\item G is precompact.
\end{enumerate}
Then A is $r$-weakly compact-bounded in G.
\end{prop}

\end{document}